\documentclass[12pt,reqno]{amsart}

\def ~{\hspace{1mm}}

\newcommand{\Nbar}{\underline{N}}
\newcommand{\Vbar}{\underline{V}}
\newcommand{\Wbar}{\underline{W}}

\newcommand{\Tbar}{\underline{T}}
\newcommand{\deltabar}{\underline{\delta}}
\newcommand{\xunderbar}{\underline{x}}
\newcommand{\xoverbar}{{\overline{x}}}
\newcommand{\zoverbar}{{\overline{z}}}
\newcommand{\zbar}{{\overline{z}}}
\newcommand{\Dbar}{\overline{\mathbb D}}
\newcommand{\Ebar}{\overline{E}}
\newtheorem{thm}{Theorem}[section]
\newtheorem{cor}[thm]{Corollary}
\newtheorem{lem}[thm]{Lemma}

\newtheorem{defn}[thm]{Definition}
\newtheorem{rem}[thm]{Remark}

\numberwithin{equation}{section}

\def\textmatrix#1&#2\\#3&#4\\{\bigl({#1 \atop #3}\ {#2 \atop #4}\bigr)}
\def\dispmatrix#1&#2\\#3&#4\\{\left({#1 \atop #3}\ {#2 \atop #4}\right)}

\begin{document}
\title[The tetrablock as a spectral set]
{The tetrablock as a spectral set}

\author[Bhattacharyya]{Tirthankar Bhattacharyya}
\address[Bhattacharyya]{Department of Mathematics, Indian Institute of Science, Banaglore 560 012}
\email{tirtha@member.ams.org}

\thanks{The author's research is supported by Department of Science and
Technology, India through the project numbered SR/S4/MS:766/12 and
University Grants Commission, India via DSA-SAP}

\begin{abstract}

We study  a commuting triple of bounded
operators $(A, B, P)$ which has the tetrablock as a
spectral set.
\end{abstract}

\maketitle

\section{Introduction}

We follow Arveson's terminologies from \cite{sub2}.

A  compact subset $K$ of $\mathbb C^n$ is called a spectral set
for a commuting $n$-tuple of operators $\deltabar = (\delta_1,
\delta_2, \ldots ,\delta_n)$ acting on a Hilbert space $\mathcal
H$ if the Taylor joint spectrum of $\deltabar$ is contained in $K$
and
\begin{equation} \label{vN} \| r(\delta_1, \delta_2, \ldots , \delta_n) \| \le \| r \|_{\infty, K} =
\sup \{ | r(z_1, z_2, \ldots ,z_n) | : (z_1, z_2, \ldots ,z_n) \in
K \} \end{equation} for any rational function $r$ in rat$(K)$.

For an $m \times m$ matrix valued function $r$, we can make sense
of  $r(\delta_1, \delta_2, \ldots \delta_n)$. Any such $r$ is a
matrix of functions of the form
$$ r = [ r_{ij} ]_{i,j=1}^m$$
and hence
$$r(\delta_1, \delta_2, \ldots \delta_n) = [ r_{ij} (\delta_1, \delta_2, \ldots \delta_n)]_{i,j=1}^m.$$
This is then a block operator matrix acting as an operator on the
direct sum of $n$ copies of $\mathcal H$. The set $K$ is called a
complete spectral set if the von Neumann type inequality
(\ref{vN}) above holds for all rational functions taking values in
matrices of any order with the norm now being
$$ \| r \|_{\infty, K} =
\sup \{ \| r(z_1, z_2, \ldots ,z_n) \| : (z_1, z_2, \ldots ,z_n) \in K \}.$$

Say that $K$ has the property \textbf P if the following holds.
"If $K$ is a spectral set for a commuting tuple $\deltabar$, then
it is a complete spectral set for $\deltabar$." It has
historically intrigued operator theorists which subsets of the
complex plane or of higher dimensional Euclidean space have this
property.

\begin{enumerate}

\item In dimension one, the unit disk (von Neumann \cite{vN})
    has this property. Berger, Foias and Lebow extended that
    result to any simply connected domain. Whether a domain in
    the plane with a single hole enjoyed the property \textbf
    P was open for a long time till Agler showed in
    \cite{agler} that the annulus has this property. Dritschel
    and McCullough proved in \cite{DM} that a multiply
    connected domain in general does not. They use
    Gel'fand-Naimark-Segal construction combined with a
    cone-separation argument. Agler, Harland and Raphael in
    \cite{ahr} have shown that there exists a planar domain
    with two holes and an operator $\delta$ on a Hilbert space
    of dimension $4$ for which the property \textbf P fails.
    Thus, this branch of generalization from the disc to
    multiply connected domain is now complete.

    \item The other direction of generalization is to higher
        dimensions. In dimension two, the bidisc (Ando
        \cite{Ando-bidisc}) and the symmetrized bidisc have
        this property, see Agler and Young \cite{ay-jfa} -
        \cite{ay-jot}.

    \item No subset of $\mathbb C^3$ with property \textbf P
        is known.

    \item If $K$ is the closed unit ball of some norm on
        $\mathbb C^n$ where $n > 2$, then $K$ cannot have this
        property. See Paulsen \cite{Paulse-JFA} and Pisier
        \cite{Pisier-os}.

        \end{enumerate}

Theorem 1.2.2 and the remark following it in Arveson \cite{sub2}
gives a novel idea that has been taken up and pursued heavily by
operator theorists. A simultaneous dilation of $\deltabar$
consists of $n$ commuting bounded operators $\Delta_1, \Delta_2,
\ldots ,\Delta_n$ on a Hilbert space $\mathcal K$ containing
$\mathcal H$ in such a way that
\begin{equation} \label{dilation} P_{\mathcal H} \Delta_1^{k_1} \Delta_2^{k_2} \ldots \Delta_n^{k_n}
|_{\mathcal H} = \delta_1^{k_1} \delta_2^{k_2} \ldots
\delta_n^{k_n} \end{equation} for all non-negative integers $k_1,
k_2, \ldots ,k_n$. The dilation is called minimal if
$$ \mathcal K = \overline{\mbox{span}} \{\Delta_1^{k_1} \Delta_2^{k_2} \ldots \Delta_n^{k_n} h :
k_1, k_2, \ldots , k_n \mbox{ are non-negative integers and } h
\in \mathcal H \}.$$ Dilation was a game-changing new geometric
concept introduced by Sz.-Nagy and has made effective yet
unpredictable appearances in many places. For more details, see
the classic \cite{nazy}. Naturally, one wants the dilation
operators to be nicer than the given ones. The dilation
$(\Delta_1, \Delta_2, \ldots ,\Delta_n)$ is called a normal
dilation if its Taylor joint spectrum is contained in
the Shilov boundary of $K$ (relative to $A(K)$) and the $\Delta_i$
are normal operators. Arveson showed that $K$ is a complete
spectral set for $\deltabar$ if and only if $\deltabar$ has a
normal dilation.

Similarities and dissimilarities of the symmetrized bidisc and the
tetrablock have attracted recent attention from both complex
analysts and operator theorists. The geometry of them has been
studied by Edigarian, Kosinski and Zwonek in
 \cite{EZ}, \cite{EKZ}, \cite{Kos} and \cite{Zwo}.  Young's work on the symmetrized bidisc,
 the $\Gamma$-contractions and the tetrablock with his many co-authors
 (see \cite{awy}, \cite{awy-cor}, \cite{ay-jfa}, \cite{ay-ems}
 and \cite{ay-jot}) has been from opetator theoritic point of view.

The symmetrized bidisc
$$ G = \{ (z_1 + z_2 , z_1 z_2 ) : z_1, z_2 \in \mathbb D \} $$
and the tetrablock
$$ E = \{ (a_{11}, a_{22}, \det A ) : A = \dispmatrix a_{11} &
a_{12} \\ a_{21} & a_{22} \\ \mbox{ is in the classical Cartan domain of type II}\}$$ are both non-convex domains. The Lempert function and the Caratheodory distance coincide on them. But they cannot be exhausted by domains biholomorphic to convex ones.

Those are the similarities. The dissimilarities appear in operator
theory on these domains. Normal dilation was
constructed for a pair of commuting operators, having the
symmetrized bidisc as a spectral set, in \cite{bpsr}. Now we show
that given a triple of commuting operators, having the tetrablock
as a spectral set, its fundamental operators need to satisfy certain commutativity conditions so that there is a normal dilation.

\section{Preliminaries}

\begin{defn} \label{contraction}

Let $(A, B, P)$ be a triple of commmuting bounded
operators on a Hilbert space $\mathcal H$. We call it a tetrablock contraction
if $\Ebar$ is a spectral set for $(A, B, P)$.  \end{defn}

A tetrablock contraction consists of commuting contractions. This can be seen easily if we take the
polynomial $f_i(x_1, x_2, x_3) = x_i$ for any $i=1,2,3$. From the definition of $E$, it is evident that $\|
f_i \|_{\infty, \Ebar} \le 1$. Thus $\| A \| = \| f_1 ((A, B, P)) \| \le 1$. Similarly, $B$ and $P$ are contractions. Let $D_P = ( I - P^* P)^{1/2}$ be the defect operator.

 An added source of pleasure consists of numerous relations of the tetrablock to the
 symmetrized bidisk and of tetrablock contractions to $\Gamma$-contractions. The closed symmetrized bidisc will be  denoted by $\Gamma$. This is a polynomially convex subset of $\mathbb C^2$ with the distinguished boundary (the Shilov boundary relative to $A(\Gamma)$) being
$$ b\Gamma = \{ (z_1 + z_2 , z_1 z_2 ) : | z_1 | = | z_2 | = 1 \}.$$
A commuting pair of bounded operators $(S,P)$ is called

\begin{enumerate}

\item a $\Gamma$-contraction if the closed
symmetrized bidisc $\Gamma$ is a spectral set for $(S,P)$,

\item a $\Gamma$-unitary if $S$ and $P$ are normal operators
with the Taylor joint spectrum of the pair $(S,P)$ contained in $b\Gamma$,

\item a $\Gamma$-isometry if it is the restriction of a $\Gamma$-unitary to a joint invariant subspace.
\end{enumerate}

Agler and Young studied the $\Gamma$-contractions in a series of papers, see \cite{ay-jfa} - \cite{ay-jot}.
They proved that if $\Gamma$ is a spectral set for $(S,P)$, then it is a complete spectral set for $(S,P)$. Thus, by Arveson's theorem, a $\Gamma$-contraction has a normal dilation. In other words, a $\Gamma$ contraction has a $\Gamma$ unitary dilation. The change in point of view introduced in \cite{bpsr} explicitly constructed a $\Gamma$-isometric dilation of a $\Gamma$-contraction by finding solution of an operator
equation. We pick out the salient features from there as well as from earlier papers of Agler and Young (\cite{ay-jot} and \cite{ay-jfa}) which we summarize below. For a contraction $P$ and bounded operator $S$ commuting with $P$, define $ \rho( S,P) =
2(I-P^*P)-(S-S^*P)-(S^*-P^*S)$.

\begin{thm} \label{oldtheorem}
Let $(S,P)$ be a pair of commuting bounded operators on $\mathcal H$ with $P$ being a contraction.
Then the following are equivalent.

 \begin{enumerate}

 \item $\Gamma$ is a spectral set for the commuting pair
     $(S,P)$.

 \item $I - P^* P \ge {\mbox Re } \beta (S - S^*P)$ for all
     $\beta$ on the unit circle. In other words, $\rho (\beta
     S , \beta^2 P) \ge 0$ for all $\beta$ on the unit circle.
 \item The operator equation
$$ S - S^*P = (I - P^* P)^{1/2} \Phi (I - P^* P)^{1/2}$$
called its fundamental equation has a unique solution $\Phi \in
\mathcal B \; (\overline{\mbox{Ran}} (I - P^* P)^{1/2})$ with
numerical radius of $\Phi$ being not greater than one.
\item There is a pair of commuting normal operators $R$ and
    $U$ on a bigger Hilbert space $\mathcal K$ containing $\mathcal H$
    such that the Taylor joint spectrum $\sigma (R, U)$ is contained in the
    distinguished boundary of $\Gamma$ and
    $$ P_{\mathcal H} R^m U^n |_{\mathcal H} = S^m P^n$$
    for all non-negative integers $m$ and $n$.
\end{enumerate}
\end{thm}

Explicit construction of the $\Gamma$ unitary $(R,U)$ involves the fundamental
operator $\Phi$ obtained above. Certain routine facts will be used without
further ado. Some of these are

\begin{enumerate}

\item if $(S,P)$ is a $\Gamma$-contraction and $P$ is a unitary,
then $(S,P)$ is a $\Gamma$-unitary,

\item  if $(S,P)$ is a $\Gamma$-contraction and $P$ is an isometry,
then $(S,P)$ is a $\Gamma$-isometry,

\end{enumerate}

\section{The fundamental equations for a tetrablock contraction}

The main content of this section is Theorem \ref{chain} which
solves the fundamental equations for a tetrablock contraction.
The tetrablock $E$ has several characterizations as
stated below from the papers \cite{awy} and \cite{awy-cor}. It follows from the
definition of $E$ that $ 1 - x_1 z \neq 0 \neq 1 - x_2 z$ for any
$\xunderbar \in E$ and any $z \in \Dbar$. Thus the holomorphic
functions
\begin{equation} \Psi (z , \xunderbar ) = \frac{x_3 z - x_1}{x_2 z - 1}
\mbox{ and } \Upsilon (z , \xunderbar ) = \frac{x_3 z - x_2}{x_1 z
- 1} \label{functions} \end{equation} can be defined on $\mathbb D
\times E$. Let $\mathbb M _2$ denote the algebra of all $2 \times
2$ complex matrices equipped with the operator norm $\| \cdot \|$.
Define $\pi : \mathbb M _2 \rightarrow \mathbb C ^3$ by

$$ \pi \left( \dispmatrix a_{11} & a_{12} \\ a_{21} & a_{22} \\ \right) = (a_{11}, a_{22}, \det A).$$
We quote from Abouhajar, White and Young a number of ways for
deciding membership of an $\xunderbar$ in $E$.

\begin{thm}[Abouhajar, White, Young] \label{AWY}

For $\xunderbar = (x_1, x_2, x_3)$ in $\mathbb C ^3$, the
following are equivalent. And they all guarantee membership in $E$
(respectively $\in \Ebar$).

\begin{enumerate}

\item $ 1 - x_1 z - x_2 w + x_3 zw \neq 0$ whenever $| z | <
    1$ and $| w | < 1$,

\item $\| \Psi ( \cdot , \xunderbar ) \|_{H^\infty} < 1$ (respectively $\leq 1$),

\hspace*{-0.5in} (2') $\| \Upsilon ( \cdot , \xunderbar ) \|_{H^\infty} < 1$ (respectively $\leq 1$),

\item $| x_1 - \xoverbar_2 x_3 | + |x_1 x_2 - x_3 | < 1 - |x_2|^2$ (respectively $\leq 1 - |x_2|^2$),

\hspace*{-0.5in} (3') $| x_2 - \xoverbar_1 x_3 | + |x_1 x_2 - x_3 | < 1 - |x_1|^2$ (respectively $\leq 1 - |x_1|^2$),

\item $|x_1|^2 - |x_2|^2 + |x_3|^2 + 2 |x_2 - \xoverbar_1 x_3
    | < 1$ (respectively $\leq 1$) and $|x_2| < 1$
    (respectively $\leq 1$),

\hspace*{-0.5in} (4') $ - |x_1|^2 + |x_2|^2 + |x_3|^2 + 2 |x_1
- \xoverbar_2 x_3 | < 1$ (respectively $\leq 1$) and $|x_1| <
1$ (respectively $\leq 1$),

\item $|x_1|^2 + |x_2|^2 - |x_3|^2 + 2 |x_1x_2 - x_3 | < 1$
    (respectively $\leq 1$) and $|x_3| < 1$ (respectively
    $\leq 1$),

\item $ | x_1 - \xoverbar_2 x_3 | + |x_2 - \xoverbar_1 x_3 | < 1 - |x_3|^2$ (respectively $\leq 1 - |x_3|^2$),

\item $\xunderbar = \pi(A)$ for an $A \in \mathbb M _2$ with $\| A \| < 1$ (respectively $\leq 1$),

\item $\xunderbar = \pi(A)$ for a symmetric $A \in \mathbb M _2$ with $\| A \| < 1$ (respectively $\leq 1$),

\item $| x_3 | < 1$ (respectively $\leq 1$) and there are complex numbers
$\beta_1$ and $\beta_2$ with $| \beta_1 | + |\beta_2 | < 1$ (respectively $\leq 1$) such that
    $$ x_1 = \beta_1 + \overline{\beta}_2 x_3 \mbox{ and }  x_2 = \beta_2 + \overline{\beta}_1 x_3.$$

\end{enumerate}

\end{thm}

The theorem above will be crucial for the purpose of this note.
And we have a neat new criterion for membership of an $(x_1, x_2, x_3)$ in $E$.

\begin{lem}
\label{neat}
The triple $(x_1, x_2, x_3)$ is in $E$ if
and only if the pair $(x_1 + z x_2 , zx_3)$ is in $G$
 for every $z$ on the unit circle.

\end{lem}

\begin{proof} Theorem 1.1 in \cite{ay-jot} gives several characterizations of when a pair $(s,p)$ will be in
$G$. Of those characterizations, the one most suitable for our
present purpose is that

\begin{equation} \label{sp} | s - \overline{s} p | < 1 - |p|^2 \mbox{ and } |s| < 2. \end{equation}
Let $s_z = x_1 + z x_2$ and $p_z = z x_3$. Let $(x_1, x_2, x_3)$
be in $E$. By part (9) of Theorem \ref{AWY}, we have $| s_z | <
2$. Now
$$ | s_z - \overline{s_z} p_z | = | x_1 + zx_2 - (\overline{x_1} +
\overline{z} \overline{x_2}) z x_3 | = | x_1  -  \overline{x_2}
x_3 + z(x_2 - \overline{x_1} x_3) | \le | x_1  -
\overline{x_2} x_3 | + | x_2 - \overline{x_1} x_3 |.$$
Now appeal to part (6) of Theorem \ref{AWY} above to complete the
proof that $(s_z , p_z)$ is in $G$.

Conversely, let $(s_z , p_z)$ be in $G$ (respectively $\Gamma$)
for all $z$ on the circle. By the characterization (\ref{sp}), we
have a function $\beta$ on the circle which satisfies
$$ s_z - \overline{s_z} p_z = \beta(z) (1 - | p |^2) \mbox{ and } | \beta (z) | < 1.$$
Now
$$ \beta(z) = \frac{ s_z - \overline{s_z} p_z}{1 - | p |^2}
= \frac{ x_1 + z x_2 - \overline{x_1 + z x_2} z x_3}{1 - |x_3|^2}
= \frac{(x_1 - \overline{x_2}x_3) + z (x_2 \overline{x_1} x_3)}{1 - |x_3|^2}.$$
Since this holds for all $z$ on the circle, we have $\beta(z) = \beta_1 + z\beta_2$, where
$$ \beta_1 = \frac{x_1 - \overline{x_2}x_3}{1 - |x_3|^2} \mbox{ and } \beta_2 = \frac{x_2 - \overline{x_1}x_3}{1 - |x_3|^2}.$$
Clearly,
$$ x_1 = \beta_1 + \overline{\beta}_2 x_3 \mbox{ and }  x_2 = \beta_2 + \overline{\beta}_1 x_3 .$$
Moreover, for any $z_1$ and $z_2$ on the unit circle, we have
$$ | z_1 \beta_1 + z_2 \beta_2 | = |z_1| | \beta_1 + \frac{z_2}{z_1} \beta_2 | = |\beta(\frac{z_2}{z_1} )| < 1 .$$
Now if we choose $z_1$ and $z_2$ so that $ z_1 \beta_1 + z_2
\beta_2 = |\beta_1| + |\beta_2|$ and apply part (9) of Theorem
\ref{AWY}, that finishes the proof.
\end{proof}
The first thing to observe about tetrablock contractions is that the defining criterion can be greatly simplified.

\begin{lem} \label{simpler}

A commuting triple of bounded operators $(A, B, P)$ is a tetrablock contraction if and only if \begin{equation} \label{pT} \| f (A, B, P) \| \le \| f \|_{\infty, \Ebar} = \sup \{ | f(x_1,x_2,x_3) | : (x_1,x_2,x_3) \in \Ebar \}\end{equation}
for any holomorphic polynomial $f$ in three variables.
\end{lem}

\begin{proof}

If $(A, B, P)$ is a tetrablock contraction, then of course (\ref{pT}) just follows from definition.

The converse proof can be easily done by using polynomial
convexity of $\Ebar$. Indeed, if the Taylor spectrum $\sigma((A,
B, P))$ is not contained in $\Ebar$, then there is a point
$(\lambda_1, \lambda_2, \lambda_3)$ in $\sigma ((A, B, P))$ that
is not in $\Ebar$. By polynomial convexity of $\Ebar$, there is a
polynomial $f$ such that $ | f(\lambda_1, \lambda_2, \lambda_3) |
> \| p \|_{\infty, \Ebar}$. By polynomial spectral mapping
theorem,
$$ \sigma (f((A, B, P))) = \{ f(x_1, x_2, x_3 ) : (x_1, x_2, x_3 )
\in \sigma ((A, B, P)) \}$$ and hence the spectral radius of
$f((A, B, P))$ is bigger than $\| f \|_{\infty, \Ebar}$. But then $ \|
f((A, B, P) )\| > \| f \|_{\infty, \Ebar}$, contradicting the fact
that $\Ebar$ is a spectral set for $(A, B, P)$.

By polynomial convexity of $\Ebar$, a triple satisfying (\ref{pT}) will also satisfy
$$ \| f(A, B, P) \| \le \| f \|_{\infty, \Ebar}$$
for any function holomorphic in a neighbourhood of $\Ebar$.
Indeed, Oka-Weil theorem (Theorem 5.1 of \cite{oka-weil}) allows
us to approximate $f$ uniformly by polynomials. Then Theorem 9.9
of Chapter III of \cite{vasilescu} about functional calculus in
several commuting operators seals the rest of the deal.
\end{proof}

As a matter of detail, we note that the restriction of a tetrablock contraction to a joint invariant subspace
is a tetrablock contraction. Indeed, if $\mathcal M$ is a such a subspace,
and $f$ is any polynomial in three variables, then
$$ \| f(A|_{\mathcal M}, B|_{\mathcal M}, P|_{\mathcal M}) \| =  \| f((A, B, P))|_{\mathcal M}  \| \le
 \| f((A, B, P))  \| \le \| f \|_{\infty , \Ebar}.$$
The other thing that follows immediately from the lemma above is
that the adjoint triple $(A^*, B^*, P^*)$ is a tetrablock
contraction too.

The numerical radius $w(T)$ of a bounded operator $T$ features
naturally. It is defined as
$$ w(T) = \sup_{\| x \| \le 1} \mid \langle Tx, x\rangle \mid. $$
It is well known that $w(T)$ and $\| T \|$ define equivalent norms. Indeed,
$$ r(T) \le w(T) \le \| T \| \le 2 w(T)$$
for all bounded operators $T$ where $r(T)$ denotes the spectral
radius of $T$.

The main result of this section is the following theorem which
gives a chain of one way implications. The main content of the theorem
is the existence and uniqueness of soulutions of the fundamental
equations for a given tetrablock contraction.

\begin{defn}
For a commuting triple of contractions $\Tbar = (T_1, T_2, T_3)$,
the equations
\begin{equation} T_1 - T_2^* T_3 = D F_1 D, \mbox{ and } T_2 - T_1^* T_3 = D F_2 D \label{fe}
\end{equation}
are called the first fundamental equation and the second fundamental equation respectively.
\end{defn}

Given two bounded operators $X$ and $Y$, let $[X, Y]$ denote the commutator $XY - YX$.

\begin{thm} \label{chain} Let $A, B$ and $P$ be three commuting contractions
on a Hilbert space $ \mathcal H $. Then, in the following, $(1)
\Rightarrow (2) \Rightarrow (3) \Rightarrow (4)$.

\begin{enumerate}

\item
The triple $(A, B, P)$ is a tetrablock contraction.

\item The operator functions $\rho_1$ and $\rho_2$ defined by
$$ \rho_1((A, B, P) ) = I - P^* P + ( B^* B - A^* A ) - 2
\mbox{ Re } (B - A^* P) $$ and
$$ \rho_2((A, B, P) ) = I - P^* P + ( A^* A - B^* B ) - 2
\mbox{ Re } (A - B^* P). $$ satisfy
$$\rho_1 (A, zB, zP) \ge 0 \mbox{ and }\rho_2 (A, zB, zP) \ge 0 \mbox{ for all }z \in \Dbar.$$

\item For any $z \in \mathbb C$, if we define a pair of operators by
\begin{equation} \label{szpz} S_z = A + z B \mbox{ and }P_z = z P. \end{equation}
then $(S_z,P_z)$ is a $\Gamma$-contraction for every $z$ on the unit circle.

\item The fundamental equations (\ref{fe}) have unique solutions
    $F_1$ and $F_2$ in $\mathcal B ( \overline{\mbox Ran} D_P )$.
    Moreover, the $ \mathcal B ( \overline{\mbox Ran} D_P ))$ valued
    function $F_1 + z F_2$ has numerical radius not greater than
    $1$ for all $z \in \Dbar$.

    \end{enumerate}
    \end{thm}

\begin{proof}
$(1) \Rightarrow (2)$: It is enough to  prove that $\rho_1 (A, zB, zP) \ge 0$ for all $z \in \Dbar$. The proof for $\rho_2$ is the same.

Consider the function $\Psi$ defined in \ref{functions}. If $z \in
\mathbb D$, then $\Psi ( z , \cdot ) $ is a holomorphic function
on $E$ with
$$ \Psi ( z , (A, B, P) ) = (z P - A )(I - zB )^{-1}.$$
Because $\Ebar$ is a spectral set for $(A, B, P)$, we know that $ \|
\Psi ( z , (A, B, P) ) \| \le 1$ which in other words means that
$$  ( I - \zbar B^* )^{-1} ( \zbar P^* - A^* ) ( z P - A )
( I - z B )^{-1} \le  I $$
which translates to
$$ ( \zbar P^* - A^* ) ( z P - A ) \le  (I -
\zbar B^* )( I - z B) $$
and finally
$$  | z |^2 P^* P - \zbar P^* A - z A^* P +
A^* A \le  I - z B - \zbar B^* + | z |^2 B^*B$$ which is
nothing but $\rho_1 (A, z B, zP) \ge 0$. Since this holds
for all $z$ in the disk and since the function $\rho_1$ is
continuous, the inequality holds on $\Dbar$. That finishes the
proof of this step.

Lemma \ref{neat} showed that the terablock is
intimately connected with the symmetrised bidisc. In the
following, we shall see that the same is true for tetrablock
contractions and $\Gamma$-contractions. Many facets of the theory
of $\Gamma$-contractions will be used in this paper.

$(2) \Rightarrow (3)$:  One of the criteria for $(S_z,P_z)$ to be a $\Gamma$-contraction is that
$\rho(\alpha S_z , \alpha^2 P_z) \ge 0$ for all $\alpha$ on the unit
circle, see Theorem \ref{oldtheorem}.  Since we have
$$ \rho_1(A, z_1 B, z_1 P) \ge 0 \mbox{ and } \rho_2(A, z_2 B, z_2 P) \ge 0 $$
for all $z_1 $ and $z_2$ on the circle, adding these two, we get
$$ D_P^2  \ge  \mbox{ Re } ( z_1 \Sigma_1 + z_2 \Sigma_2 )
 =  \mbox{ Re } ( z_1 (\Sigma_1 + \overline{z_1} z_2 \Sigma_2
)).$$
So for all $\alpha$ and $z$ on the circle. we have
\begin{eqnarray*}
D_P^2 & \ge & \mbox{ Re } ( \alpha ( \Sigma_1 + z \Sigma_2 )) \\
& = & \mbox{ Re } ( \alpha ( A - B^* P + z ( B - A^* P )))\\
& = & \mbox{ Re } ( \alpha ( ( A + z B ) - (z A^* + B^*)P )) \\
& = & \mbox{ Re } ( \alpha ( ( A + z B ) - z ( A^* + \overline{z} B^*)P )) \\
& = & \mbox{ Re } ( \alpha ( ( A + z B ) - (A + z  B)^* z P )) . \end{eqnarray*}
Thus we have
$$ D_{P_z}^2 \ge \mbox{ Re } \alpha ( S_z - S_z^* P_z )$$
for all $\alpha$ on the circle which is the same as saying that
$\rho(\alpha S_z , \alpha^2 P_z) \ge 0$ for all $\alpha$ on the
unit circle. That finishes the proof.

$(3) \Rightarrow (4)$: The uniqueness part is the simplest.
Indeed, let $F_1$ and $F_1^\prime$ be two bounded operators on $
\overline{\mbox Ran} D_P $ both of which satisfy the first
fundamental equation $ \Sigma_1 = D_PXD_P$. Then $ F = F_1 -
F_1^\prime$ satisfies $D_PFD_P = 0$. Now, for $x$ and $y$ in $H$, we
have $\langle FD_P x , D_P y \rangle = \langle D_P F D_P x , y \rangle = 0$
and hence $F_1 - F_1^\prime = 0$. The existence part of this prof
starts from the fact that $(S_z, P_z )$ is a $\Gamma$-contraction
for any $z$ on the circle. Appeal to Theorem \ref{oldtheorem}
again. The $\Gamma$-contraction $(S_z, P_z )$ has a fundamental
operator, call it $F(z)$. It satisfies
$$ S_z - S_z^* P_z  =  D_{P_z} F(z) D_{P_z} \mbox{ and } w(F(z))
\le 1 $$ where $w$ stands for the numerical radius. Recalling what $S_z$ and $P_z$ are, the last equation
becomes
$$ \Sigma_1 + z \Sigma_2 = D_P F(z) D_P.$$
This holds for all $z$ on the unit circle. Integrating over the
unit circle, we see that $\Sigma_1 = D_P F_1 D_P$ where $F_1$ is the
integration of the function $F$ over the unit circle. Thus $$ z
\Sigma_2 = D_P (F(z) - F_1) D_P$$ for all $z$ on the unit circle and
hence putting $z=1$, we get that $\Sigma_2 = D_P F_2 D_P$ where $F_2 =
F(1) - F_1$. Thus we see that $F$ is a linear function $F(z) = F_1
+ z F_2$ for some bounded operators $F_1$ and $F_2$ on
$\overline{\mbox Ran} D_P $ and
$$  \Sigma_1  = D_P F_1 D_P \mbox{ and } \Sigma_2 = D_P F_2 D_P.$$ \end{proof}

\begin{rem}

By a computation similar to part (3), we have that $(zA + B, zP)$ is a $\Gamma$-contraction for $z$ on the unit circle. This can also be seen as a consequence of part (4). Indeed,
$$ zA + B - (zA + B)^* zP = zA + B - (\zbar A^* + B^*) zP = z(A - B^*P) + B - A^*P = D_P (zF_1 + F_2)D_P$$
thereby showing that the pair $(zA + B, zP)$ satisfies fundamental equation which is a necessary and sufficient condition for it to be a $\Gamma$-contraction.

\end{rem}

In the next section, we shall obtain a characterization of the fundamental operators.

\section{New results about $\Gamma$-contractions}

This section first proves new results about $\Gamma$ contractions. Then these results are applied to tetrablock contractions.

The fundamental operator $\Phi$ of a $\Gamma$-contraction $(S,P)$ is the unique bounded operator on
${\mathcal D}_P$  that satisfies $S - S^*P = D_P \Phi D_P$. The lemma below gives a
different characterization of $\Phi$.

\begin{lem}
The fundamental operator $\Phi$ of a $\Gamma$-contraction $(S,P)$ is the unique bounded linear
operator on ${\mathcal D}_P$ that satisfies the operator equation
\begin{equation} D_P S = X D_P + X^* D_P P. \label{alteqn} \end{equation}
\end{lem}

\begin{proof}

To see that $\Phi$ satisfies the equation \ref{alteqn}, we shall use a relation which originally appeared in the
proof of Theorem 4.3 of \cite{bpsr}. The relation says that if $\Phi$ is the fundamental
operator of a $\Gamma$-contraction $(S,P)$, then
\begin{equation} \label{relation} D_P S = \Phi D_P + \Phi^* D_P P. \end{equation}
The proof of this relation is simple. Let
$G=\Phi^*{D}_PP+\Phi{D}_p-{D}_pS $. Then $G$ is defined from
$\mathcal{H}\rightarrow\mathcal{D}_P$. Since $\Phi$ is the
fundamental operator of the $\Gamma$-contraction $(S,P)$, we have
$$ {D}_PG = {D}_P\Phi^* {D}_P P + {D}_P \Phi {D}_P - {{D}_P}^2S =(S^*-P^*S)P+(S-S^*P)-(I-P^*P)S=0.$$ Now $\langle
Gh,{D}_Ph' \rangle=\langle {D}_PGh,h' \rangle=0 \quad \mbox{for
all }h,h'\in\mathcal{H}$. This shows that $G=0$ and hence $\Phi^*
{D}_P P + \Phi {D}_P= {D}_PS$.

Conversely, let $X$ satisfy $ D_P S = X D_P + X^* D_P P.$ We need to show that $X = \Phi$.
Since we just proved that $\Phi$ satisfies the equation, we have $ \Phi {D}_P + \Phi^* {D}_P P =
{D}_P S = X D_P + X^* D_P P.$ Consequently, $(X - \Phi) D_P + (X - \Phi)^* D_P P = 0$.
Let $Y = X - \Phi$. So $YD_P + Y^* D_P P = 0$. We need to show that $Y = 0$. We have
\begin{eqnarray}
& & YD_P + Y^* D_P P  = 0 \nonumber \\
\mbox{ or } & & Y D_P = -Y^*D_P P \nonumber \\
\mbox{ or } & & D_P Y D_P = - D_P Y^* D_P P = P^* D_P Y D_P P = P^{*2} D_P Y D_P P^2 = \cdots \nonumber
\end{eqnarray}
Thus we have \begin{equation} D_P Y D_P = P^{*n} D_P Y D_P P^n
\label{iterate} \end{equation} for all $n=1,2, \ldots$. Now
consider the series
\begin{eqnarray*}
\sum_{n=0}^\infty \| D_P P^n h\|^2 & = & \sum_{n=0}^\infty \langle D_P P^n h , D_P P^n h \rangle \\
& = & \sum_{n=0}^\infty \langle P^{*n} D_P^2 P^n h , h \rangle \\
& = & \sum_{n=0}^\infty \langle P^{*n} (I - P^*P) P^n h , h \rangle \\
& = & \sum_{n=0}^\infty \langle (P^{*n} P^n - P^{*n+1} P^{n+1}h , h \rangle \\
& = & \sum_{n=0}^\infty ( \| P^nh \|^2 - \| P^{n+1} h \|^2 ) \\
& = & \| h \|^2 - \lim_{n\rightarrow \infty} \| P^n h\|^2.
\end{eqnarray*}
Now $\| h \| \ge \| Ph \| \ge \| P^2 h\| \ge \cdots \ge \| P^n h \| \ge \cdots \ge 0$.
So $\lim_{n\rightarrow \infty} \| P^n h\|^2$ exists. So the series is convergent.
So $ \lim_{n\rightarrow \infty} \| D_P P^n h\|^2 =0$. So
\begin{eqnarray*}
\| D_P Y D_P h \| & = & \| P^{*n} D_P Y D_P P^n h \| \mbox{ by } (\ref{iterate}) \\
& \le & \| P^{*n} \| \| D_p Y \| \| D_P P^n h \| \le \| D_p Y \| \| D_P P^n h \| \rightarrow 0. \end{eqnarray*}
So $ D_P Y D_P= 0$. So $Y = 0$. \end{proof}

\begin{cor} \label{twoneweqns}
The fundamental operators $F_1$ and $F_2$ of a tetrablock contraction $(A, B, P)$
are the unique bounded linear operators on the range closure of $D_P$ that satisfy the pair of operator equations
$$ D_P A = X_1 D_P + X_2^* D_P P \mbox{ and } D_P B = X_2 D_P + X_1^* D_P P.$$
\end{cor}

\begin{proof}
Consider the $\Gamma$-contraction $(S_z, P_z)$ defined in (\ref{szpz}).
Its fundamental operator is $F_1 + zF_2$. Thus by the first part of the proof the lemma above, we have
$$D_P (A + zB) = (F_1 + zF_2) D_P + (F_1 + zF_2)^* D_P zP = F_1 D_P + F_2^* D_P P + z( F_2 D_P + F_1^* D_P P)$$
because $z$ comes from the unit circle. Hence $F_1$ and $F_2$ satisfy the given operator equations.
Now take a pair $(X_1, X_2)$ that satisfies the equations. Then $X_1 + zX_2$ satisfies (\ref{alteqn})
for the $\Gamma$-contraction $(S_z. P_z)$. Indeed, remembering that $|z| = 1$, we have
\begin{eqnarray*} D_P S_z = D_P (A + zB) = D_PA + zD_PB & = & X_1 D_P + X_2^* D_P P + z( X_2 D_P + X_1^* D_P P) \\
& = & X_1 D_P + z X_2 D_P + X_2^* D_P P +  + z X_1^* D_P P \\
& = & (X_1 + z X_2) D_P + (X_1 + z X_2)^* D_P zP.\end{eqnarray*}
By uniqueness, $X_1 + zX_2 = F_1 + zF_2$. Since this holds for all $z$ on the unit circle,
we have $X_1 = F_1$ and $X_2 = F_2$. \end{proof}

\begin{lem} \label{newresult}
Let $S_1, S_2$ and $P$ be three commuting bounded operators such
that $(S_1, P)$ and $(S_2, P)$ are $\Gamma$-contractions with
commuting fundamental operators $\Phi_1$ and $\Phi_2$. Then
$$ S_1^* S_2 - S_2^* S_1 = D_P (F_1^* F_2 - F_2^* F_1)
    D_P.$$
\end{lem}

\begin{proof}

Using commutativity of $S_1$ and $S_2$, we get $S_2^*S_1^* P =
S_1^* S_2^* P$. Now we use the fundamental equation for
$\Gamma$-contractions to get that
$$ S_2^* (S_1 - D_P \Phi_1 D_P) = S_1^* (S_2 - D_P \Phi_2 D_P).$$
Thus, using (\ref{relation}), we get that $ S_2^* S_1 - S_1^* S_2$
is the same as $(D_P \Phi_2^* + P^* D_P \Phi_2) \Phi_1 D_P - (D_P
\Phi_1^* + P^* D_P \Phi_1) \Phi_2 D_P$ which is equal to $ D_P
(\Phi_2^* \Phi_1 - \Phi_1^* \Phi_2) D_P$ in view of commutativity
of the fundamental operators. That proves the lemma.

\end{proof}

\begin{cor}
\label{TandF} Let $(A, B, P)$ be a tetrablock contraction with
commuting fundamental operators $F_1$ and $F_2$. Then
$$A^* A - B^* B = D ( F_1^* F_1 - F_2^* F_2 )D.$$
\end{cor}

\begin{proof}

This follows from the Lemma \ref{newresult}. With $S_1 = A +
zB, S_2 = zA + B$ and $P = z P$, we shall get
$$ S_1^* S_2 - S_2^* S_1 = (z - \zbar) (A^* A - B^* B) \mbox{ and }
\Phi_1^* \Phi_2 - \Phi_2^* \Phi_1 = (z - \zbar)(F_1^* F_1 - F_2^* F_2).$$
That finishes the proof.
\end{proof}

 Armed with the fundamental operators $F_1$
and $F_2$ of a tetrablock contraction $(A, B, P)$, we are ready to
investigate those tetrablock contractions which are special. A
unitary operator is a special kind of contraction because it is
normal and its spectrum is contained in the unit circle. Indeed,
that is a characterization. The next section completely unravels
the structure of a commuting triple which consists of normal
operators and whose Taylor joint spectrum is contained in the
distinguished boundary of $E$.

\section{Tetrablock unitaries and tetrablock isometries}

The beginning of this section warrants a discussion on what
the distinguished boundary of a domain $\Omega$ is. Any study of a dilation involves those special tuples of operators
whose joint spectrum is contained in the distinguished boundary
$b \Omega$ of $\Omega$. Let $A(\Omega)$ be the algebra of continuous
scalar functions on $\overline{\Omega}$ that are holomorphic on $\Omega$.
A boundary for $\Omega$ is a subset $C$ of $\Omega$ such that every function
in $A(\Omega)$ attains its maximum modulus on $C$. It is well-known that for a
polynomially convex $\Omega$, there is a smallest closed boundary of
$\Omega$, contained in all the closed boundaries of $\Omega$. This is called the distinguished boundary
of $\Omega$. We need characterizations of the distinguished boundary of $\Ebar$ and this is given in \cite{awy}. We quote parts of Theorem 7.1 from there.

\begin{thm}[Abouhajar, White, Young] \label{bE}

For $\xunderbar = (x_1, x_2, x_3) \in \mathbb C^3$, the following are equivalent.

\begin{enumerate}

\item $\xunderbar \in bE$,

\item $x_1 = \xoverbar_2 x_3, |x_3| = 1$ and $|x_2| \le 1$;

\item there exists a $2 \times 2$ unitary matrix $U$ such that $x = \pi(U)$;

\item there exists a symmetric $2 \times 2$ unitary matrix $U$ such that $x = \pi(U)$;

\item $\xunderbar \in \Ebar$ and $|x_3| = 1.$

\end{enumerate} \end{thm}

Motivated by Lemma \ref{neat}, it is natural to ask whether an analogous
characterization holds for the distinguished boundary and the answer is yes.

\begin{lem} \label{neatbdry}

A triple $\xunderbar = (x_1, x_2, x_3) \in bE$ if and only if $(x_1 + zx_2 , zx_3)
\in b\Gamma$ for all $z$ from the unit circle. \end{lem}

\begin{proof}

Indeed, if $(x_1,x_2,x_3) \in bE$, then we know from Lemma \ref{neat}
that $(s_z, p_z) \in \Gamma$ and we know from Theorem \ref{bE} that $ | p_z | = 1$.
These two together form a characterization for $(s_z, p_z)$ to be in $b\Gamma$
(see Theorem 1.3 of \cite{ay-jot}). Conversely, if for every $z$ on the unit circle,
$(s_z, p_z) \in b\Gamma$, then by Lemma \ref{neat}, we have $(x_1,x_2,x_3) \in E$.
That along with the fact $| x_3 | = 1$ implies that $(x_1,x_2,x_3)$ has to be in $bE$
by part (5) of Theorem \ref{bE}. \end{proof}

We begin the study of those tetrablock contractions which are
special in the same sense that unitaries are special among
contractions. So these special tetrablock contractions are the candidates for dilation.

\begin{defn} \label{tunitary}
A tetrablock unitary is a commuting triple of normal operators
$\Nbar = (N_1, N_2, N_3)$ such that its Taylor joint spectrum
$\sigma ( \Nbar )$ is contained in $bE$.
\end{defn}
\begin{thm}
\label{tu}

Let $\Nbar = (N_1, N_2, N_3)$ be a commuting triple of bounded operators. Then the following are equivalent:

\begin{enumerate}

\item $\Nbar$ is a tetrablock unitary,

\item $N_3$ is a unitary, $N_2$ is a contraction and $N_1 = N_2^* N_3$,

\item there is a $2 \times 2$ unitary block operator matrix $[
    U_{ij} ]$ where $U_{ij}$ are commuting normal operators
    and $\Nbar = (U_{11}, U_{22}, U_{11} U_{22} - U_{21}
    U_{12})$,

\item $N_3$ is a unitary and $\Nbar$ is a tetrablock
    contraction,

    \item the family $\{ (R_z , U_z) : |z| = 1\}$ where $R_z = N_1 + zN_2$
        and $U_z = zN_3$ is a commuting family of $\Gamma$-unitaries.

\end{enumerate}
\end{thm}

\begin{proof}

$(1) \Rightarrow (2)$: By definition of a tetrablock unitary,
$N_1, N_2$ and $N_3$ are commuting normal operators and their
Taylor joint spectrum is contained in $bE$. By spectral mapping
theorem, $\sigma (N_3) = P_3 \sigma (\Nbar)$ where $P_3$ is the
projection onto the third co-ordinate. Since $\sigma(\Nbar)$ is
contained in $bE$, we have $| \lambda | = 1$ for all $\lambda \in
\sigma(N_3)$. So $N_3$ is a normal operator with its spectrum
contained in the unit circle. So it is a unitary.

Consider the $C^*$-algebra $\mathcal C$ generated by the commuting
normal operators $N_1, N_2$ and $N_3$. This commutative
$C^*$-algebra is isometrically isomorphic, by the Gelfand map, to
$C(\sigma(\Nbar))$. The Gelfand map takes $N_i$ to the co-ordinate
function $x_i$ for $i=1,2,3$. The co-ordinate functions satisfy
$x_1 = \overline{x}_2 x_3$ on the whole of $bE$ and hence on
$\sigma(\Nbar)$ which is contained in $bE$. Thus $N_1 = N_2^*
N_3$.

\vspace{3mm}

$(2) \Rightarrow (3)$: We first note that $N_2$ is normal.
Indeed, using the fact that $N_1 = N_2^* N_3$,
we get that $N_1 N_2 = N_2^* N_3 N_2 = N_2^* N_2 N_3$. On the other hand,
$N_2 N_1 = N_2 N_2^* N_3$. Since $N_1$ and $N_2$ commute, we have
$N_2^* N_2 N_3 = N_2 N_2^* N_3$. Multiplying both sides on the right by
$N_3^*$, we get $N_2$ to be normal. Now just take
$$ U = \dispmatrix N_2^* N_3 & -D_{N_2} \\ N_3 D_{N_2} & N_2 \\.$$

\vspace{3mm}

$(3) \Rightarrow (4)$: We shall verify that $\Nbar$ is a
tetrablock contraction by using Lemma \ref{simpler}. First note that,
because $U$ is a unitary, we have
$$ \dispmatrix I & 0 \\ 0 & I \\ = U^* U = \dispmatrix U_{11}^* &
U_{21}^* \\ U_{12}^* & U_{22}^* \\ \dispmatrix U_{11} & U_{12} \\ U_{21} & U_{22} \\
= \dispmatrix U_{11}^* U_{11} + U_{21}^* U_{21} & U_{11}^* U_{12} + U_{21}^* U_{22} \\
U_{12}^* U_{11} + U_{22}^* U_{21} & U_{12}^* U_{12} + U_{22}^*
U_{22} \\ . $$ So
\begin{equation} \label{U1} U_{11}^* U_{11} + U_{21}^*
U_{21} = I = U_{12}^* U_{12} + U_{22}^* U_{22} \end{equation} and
\begin{equation} \label{U2} U_{12}^* U_{11} + U_{22}^* U_{21} = 0. \end{equation} Since
$U_{11}, U_{12}, U_{21}$ and $U_{22}$ are commuting normal
operators, given any $(z_{11}, z_{12}. z_{21}, z_{22}) \in \sigma
(U_{11}, U_{12}, U_{21}, U_{22})$, we have
$$(z_{11}, z_{12}.
z_{21}, z_{22}, \zoverbar_{11}, \zoverbar_{12}, \zoverbar_{21},
\zoverbar_{22}) \in \sigma (U_{11}, U_{12}, U_{21}, U_{22},
U_{11}^*, U_{12}^*, U_{21}^*, U_{22}^*).$$ Thus by the relations (\ref{U1}) and (\ref{U2}), we have
$$ | z_{11} |^2 + | z_{21} |^2 = 1 = | z_{12} |^2 + | z_{22} |^2
\mbox{ and } \zoverbar_{12} z_{11} + \zoverbar_{22} z_{21} = 0. $$
Thus the scalar matrix $Z = \textmatrix z_{11} & z_{12} \\ z_{21}
& z_{22} \\ $ is a unitary. Let $p$ be a polynomial in three variables. Then
\begin{eqnarray*}
& & \| p(N_1, N_2, N_3) \| \\
& = & r (p(N_1, N_2, N_3)) [\mbox{by normality}] \\
& = & \sup \{ | p(\pi (Z)) | : Z = \left(
                                     \begin{array}{cc}
                                       z_{11} & z_{12} \\
                                       z_{21} & z_{22} \\
                                     \end{array}
                                   \right)
 \mbox{ with } (z_{11}, z_{12}. z_{21}, z_{22}) \in \sigma
(U_{11}, U_{12}, U_{21}, U_{22}) \}\\
& \le & \sup \{ | p(\lambda_1, \lambda_2, \lambda_3) | :
(\lambda_1, \lambda_2, \lambda_3) \in bE \} [\mbox{by the discussion above}]\\
& \le & \sup \{ | p(\lambda_1, \lambda_2, \lambda_3) | :
(\lambda_1, \lambda_2, \lambda_3) \in \Ebar \} = \| p \|_{\infty , E} \end{eqnarray*}
proving that $\Nbar$ is a tetrablock contraction.

\vspace{3mm}

$(4) \Rightarrow (5)$: From Theorem \ref{chain}, we know that $(R_z, U_z)$
is a $\Gamma$-contraction for every $z$ on the unit circle.
Moreover, $U_z$ is a unitary. A $\Gamma$-contraction whose second
component is a unitary has to be a $\Gamma$-unitary, see part (4)
of Theorem 2.5 of \cite{bpsr}. The commutativity is clear.

\vspace{3mm}

$(5) \Rightarrow (1)$: First note that $N_3$ is a unitary by putting $z=1$.
Since $R_1$ and $R_{-1}$ are commuting normal operators, $N_1 = (R_1 + R_{-1} )/2$
and $N_2 = (R_1 - R_{-1} )/2$ are commuting normal operators.
It remains to see that the joint spectrum $\sigma(\Nbar)$ is contained in $bE$.

The proof of that will depend on the observation that
$(x_1,x_2,x_3) \in bE$ if and only if for every $z$ on the unit
circle, $(s_z, p_z) \in b\Gamma$ where $s_z = x_1 + z x_2$ and
$p_z = zx_3$. Let $(x_1,x_2,x_3)$ be a point in the Taylor joint
spectrum $\sigma(\Nbar)$ of $\Nbar$. Let $z$ be from the unit
circle. Then the Taylor joint spectrum of $(R_z, U_z)$ is the set
$\{ (s_z, p_z) : s_z = x_1 + z x_2$ and $p_z = zx_3\}$ which is
contained in $b\Gamma$ because $(R_z, U_z)$ is a $\Gamma$-unitary.
Thus any point $(x_1,x_2,x_3)$ in $\sigma(\Nbar)$ has the property
that $(x_1 + z x_2 , zx_3)$ is in $b\Gamma$ for every $z$ on the
unit circle. By Lemma \ref{neatbdry} above, $(x_1,x_2,x_3)$ then
has to be in $bE$ and that completes the proof.
\end{proof}

The class of tetrablock contractions that are natural candidates for dilation
are the tetrablock unitaries for reasons that have been amply described.
However, we can simplify our lives by enlarging the class to include the following.

\begin{defn}

A tetrablock isometry is the restriction of a tetrablock unitary to a joint invariant subspace.
\end{defn}

This is also expressed by saying that a tetrablock isometry is a
triple of commuting bounded operators which has a simultaneous
$extension$ to a tetrablock unitary. Thus a tetrablock isometry
$\Vbar = (V_1, V_2, V_3)$ always consists of commuting subnormal
operators. Moreover, $V_3$ has to be an isometry. Call $\Vbar$ a
$pure$ tetrablock isometry if $V_3$ is a pure isometry, i.e., a
shift of some multiplicity. When we dilate a tetrablock
contraction, it will be enough to dilate only to a tetrablock
isometry because extension of a dilation is a dilation again as
we shall see when we define dilation in the next section.
It is for this reason that we need to understand the structure of
tetrablock isometries completely. The rest of this section does that,

\begin{thm}[Wold decomposition for a tetrablock isometry]
\label{Wold}
Let $\Vbar = (V_1, V_2, V_3)$ be a tetrablock isometry on a
Hilbert space $\mathcal H$. Then there is a decomposition of
$\mathcal H$ into a direct sum $\mathcal H = \mathcal H_1 \oplus
\mathcal H_2$ satisfying the following two conditions.

\begin{enumerate}

\item The subspaces $\mathcal H_1$ and $\mathcal H_2$ are
    reducing subspaces for each of the $V_i$.

    \item If $N_i = V_i |_{\mathcal H_1}$ and $W_i = V_i
        |_{\mathcal H_2}$, then the triple $\Nbar$ is a
        tetrablock unitary and the triple $\Wbar$ is a pure
        tetrablock isometry.
        \end{enumerate}

        \end{thm}

\begin{proof}

Let $V_3 = N_3 \oplus W_3$ be the Wold decomposition of the
isometry $V_3$ into its unitary part $N_3$ and the shift part
$W_3$. Suppose $\mathcal H = \mathcal H_1 \oplus \mathcal H_2$ be
the corresponding decomposition of the whole space $\mathcal H$.
Thus $V_3$ has the block matrix decomposition
$$ V_3 = \left(
           \begin{array}{cc}
             N_3 & 0 \\
             0 & W_3 \\
           \end{array}
         \right) .
$$
If we now write $V_2$ according to this decomposition of the
space, then let its block matrix form be
$$ V_2 = \left(
           \begin{array}{cc}
             A_{11} & A_{12} \\
             A_{21} & A_{22} \\
           \end{array}
         \right) .
$$
By commutativity of $V_2$ with $V_3$, the off diagonal entries
$A_{12}$ and $A_{21}$ end up commuting with the unitary $N_3$ and
the shift $W_3$. It is well known that no non-zero operator can do
that because $(W_3^*)^n$ converges to $0$ strongly as $n$ tends to
$\infty$. Thus $V_2$ is block diagonal too, say $V_2 = N_2 \oplus
W_2$ where $N_2$ and $W_2$ are contractions. Note that $V_1 =
V_2^* V_3$ because it inherits this property from its normal
extension. Indeed, by definition of a tetrablock isometry, there
is a Hilbert space $\mathcal K$ containing $\mathcal H$ and a
tetrablock unitary $(M_1, M_2, M_3)$ on $\mathcal K$ such that
$M_i |_{\mathcal H} = V_i$ for $i=1,2,3$. Thus for $h_1$ and $h_2$
in $\mathcal H$, we have $$ \langle V_1 h_1 , h_2 \rangle =
\langle M_1 h_1 , h_2 \rangle = \langle M_2^* M_3 h_1 , h_2
\rangle = \langle M_3 h_1 , M_2 h_2 \rangle = \langle V_3 h_1 ,
V_2 h_2 \rangle = \langle V_2^* V_3 h_1 ,  h_2 \rangle .$$
Consequently, $V_1 = N_1 \oplus W_1$ where $N_1 = N_2^* N_3$ and
$W_1 = W_2^* W_3$. By part (2) of Theorem \ref{tunitary}, we have
that $(N_1, N_2, N_3)$ is a tetrablock unitary. Thus our given
tuple $\Vbar$ has now been written as the direct sum of a
tetrablock unitary $\Nbar$ and a tuple $\Wbar$ whose third
component $W_3$ is a shift. That completes the proof.
\end{proof}

However, given a triple, how does one decide whether it is a
tetrablock isometry or not. The following result gives necessary
and sufficient criteria. The fourth part will be very handy when we construct
the dilation.

\begin{thm} \label{ti}

Let $\Vbar = (V_1, V_2, V_3)$ be a commuting triple of bounded
operators. Then the following are equivalent.

\begin{enumerate}

\item $\Vbar$ is a tetrablock isometry.

\item $\Vbar$ is a tetrablock contraction and $V_3$ is an
    isometry.

\item $V_3$ is an isometry, $V_2$ is a contraction and $V_1$
    is the same as $V_2^* V_3$.

    \item $V_3$ is an isometry, $r(V_1) \le 1$, $r(V_2) \le 1$
        and $V_1$ is the same as $V_2^* V_3$ where $r$ stands
        for spectral radius.

\end{enumerate}

\end{thm}

\begin{rem} \label{remti}

Note that post facto, (3) and (4) above imply that $V_2 = V_1^* V_3$ as well.

\end{rem}

\begin{proof}

We shall prove that $(1) \Rightarrow (2) \Rightarrow (3)
\Rightarrow (1)$ and then $(3) \Rightarrow (4) \Rightarrow (3)$.

$(1) \Rightarrow (2)$: Given a tetrablock isometry $\Vbar = (V_1,
V_2, V_3)$ on a Hilbert space $\mathcal H$, by its definition,
there is a Hilbert space $\mathcal K$ containing $\mathcal H$ and
a tetrablock unitary $\Nbar = (N_1, N_2, N_3)$ acting on $\mathcal
K$ for which $\mathcal H$ is an invariant subspace and $V_i = N_i
|_{\mathcal H}$ for $i=1,2,3$. It is then clear that $V_3$ is an
isometry because it is the restriction of the unitary $N_3$ to the
invariant subspace $\mathcal H$. It is also clear that $\Vbar$ is
a tetrablock contraction because it is the restriction of a
tetrablock contraction $\Nbar$ to an invariant subspace.

$(2) \Rightarrow (3)$: This part will require the solutions of the
fundamental equations. This is a major departure from the theory
of $\Gamma$-contractions because properties of $\Gamma$-isometries
were deduced before the fundamental equation for a
$\Gamma$-contraction was introduced. In the case of a tetrablock
isometry, it is simplest to use the fundamental equations since
existence and uniqueness of fundamental operators have already
been deduced in Section 2. Since $\Vbar$ is a tetrablock
contraction, the first fundamental equation has a solution. Since
$V_3$ is an isometry, the right hand side of that equation
vanishes and hence we have $V_1 = V_2^* V_3$ (of course, we also
have $V_2 = V_1^* V_3$ from the second fundamental equation, but
this is redundant). Contractivity of $V_2$ (or equivalently $V_1$)
holds because all components of a tetrablock contraction are
contractions.

$(3) \Rightarrow (1)$: Given a commuting triple $\Vbar = (V_1,
V_2, V_3)$ on a Hilbert space $\mathcal H$ consisting of an
isometry $V_3$, a contraction $V_2$ commuting with $V_3$ and $V_1
= V_2^* V_3$, we invoke the Wold decomposition of the isometry
$V_3$. This gives a decomposition of the space $\mathcal H$ into a
direct sum $\mathcal H_1 \oplus \mathcal H_2$ according to which
$V_3$ decomposes as the direct sum of a normal operator $N_3$ and
a shift $W_3$. We find that $V_2$, because it commutes with $V_3$,
has $\mathcal H_1$ and $\mathcal H_2$ as reducing subspaces.
Similarly, for $V_1$. This part of the argument has been detailed
above in the proof of the Wold decomposition of a tetrablock
isometry. Let $N_1, N_2, W_1$ and $W_2$ be as in that proof.
Moreover, the relations $N_1 = N_2^* N_3$ and $W_1 = W_2^* W_3$
follow because of the given relation $V_1 = V_2^* V_3$. Then the
commuting triple $\Nbar$ consists of a unitary $N_3$, a
contraction $N_2$ and $N_1 = N_2^* N_3$. Thus it is a tetrablock
unitary. Thus to show that $\Vbar$ is a tetrablock isometry, we
need to show that the tuple $\Wbar$ can be extended to a
tetrablock unitary, since the other part $\Nbar$ is already a
tetrablock unitary.

Realize $W_3$ as multiplication by the co-ordinate function $z$ on
a vector valued Hardy space $H^2(E)$ where the dimension of $E$ is
the multiplicity of the shift $W_3$. Since $W_1$ and $W_2$ commute
with this shift and with each other, there are commuting
$H^\infty(E)$ functions $\varphi_1$ and $\varphi_2$ such that $W_i
= T_{\varphi_i}$, the multiplication on $H^2(E)$ by $\varphi_i$
for $i=1,2$. Moreover, the $H^\infty$ norms of the operator valued
functions $\varphi_1$ and $\varphi_2$ are not greater than one since
$W_1$ and $W_2$ are contractions. Because $W_1={W_2}^* W_3$, or equivalently
$M_{\varphi_1}^E={M_{\varphi_2}^E}^*M_z^E$, we have
\begin{equation} \label{phis} \varphi_1(z)={\varphi_2}^*(z)z \mbox{ for all } z\in\mathbb{T}.\end{equation}
Consider on $L^2(E)$, the multiplication operators $U_{\varphi_1}^E, U_{\varphi_2}^E$
and $U_z^E$, multiplications by $\varphi_1(z)$, $\varphi_2(z)$ and $z$ respectively.
Obviously $U_z^E$ is a unitary operator on $L^2(E)$. Because of the relation
(\ref{phis}) that the functions $\varphi_1$ and $\varphi_2$ satisfy, we have
$U_{\varphi_1}^E = (U_{\varphi_2}^E)^* U_z^E$.
Altogether the triple $(U_{\varphi_1}^E, U_{\varphi_2}^E, U_z^E)$
makes a tetrablock unitary and it extends the triple $\Wbar$. So we are done.

Thus we have proved $(1) \Rightarrow (2) \Rightarrow (3)
\Rightarrow (1)$. For the remaining part, first note that $(3)$
implies $(4)$ trivially. For the converse, since we have $V_1 =
V_2^* V_3$, multiplying both sides from the left by $V_3^*$, we
get $V_3^* V_1 = V_3^* V_2^* V_3$ which by commutativity is the
same as $V_2^* V_3^* V_3$ which again is just $V_2^*$ because
$V_3$ is an isometry. Thus $V_2 = V_1^* V_3$ as well. Consider the
two operators
$$B_1 = \dispmatrix 0 & V_1 \\ V_2 & 0 \\ \mbox{ and }
B_2 = \dispmatrix V_3 & 0 \\ 0 & V_3 \\.$$ By the relations $V_1 =
V_2^* V_3$ and $V_2 = V_1^* V_3$, this pair satisfies $B_1 = B_1^*
B_2$. This immediately implies that $B_1$ is hyponormal. Indeed,
$B_1 B_1^* = B_1^* B_3 B_3^* B_1 \le B_1^* B_1$ and this is the
defining property of hyponormality. In fact, the pair $(B_1, B_2)$
is a $\Gamma$-contraction, but that is besides the point.

Now we use a remarkable theorem due to Stampfli.

\begin{thm}[Stampfli] \label{stampfli}If $X$ is a hyponormal operator, then
$\|X^n\|=\|X\|^n$ and so $\|X\|=r(X)$.
\end{thm}
For a proof of this theorem, see Proposition 4.6 of \cite{conway}.

We shall use it with $X = B_1$ to get that $\| B_1 \| = r(B_1)$.
The operator norm of $B_1$ is as follows.
$$ \| B_1 \|^2 = \| \dispmatrix 0 & V_1 \\ V_2 & 0 \\
\dispmatrix 0 & V_2^* \\  V_1^* & 0 \\ \| = \| \dispmatrix V_1
V_1^* & 0 \\ 0 & V_2V_2^* \\ \| = \max \{ \| V_1 \|^2 , \| V_2
\|^2 \}.$$ Thus $\| B_1 \| = \max \{ \| V_1 \| , \| V_2 \| \}.$
Now, for the spectral radius, we apply the spectral radius
formula. A straightforward computation using commutativity of
$V_1$ and $V_2$ shows that
$$ B_1^{2n} = \dispmatrix (V_1 V_2)^n & 0 \\ 0 & (V_1 V_2)^n \\
.$$ Consequently, $r(B_1) = \lim \| B_1^{2n} \|^{1/2n} = \lim \|
(V_1V_2)^n \|^{1/2n} = r(V_1V_2)^{1/2}$. Because $V_1$ and $V_2$
commute, the spectrum of $V_1V_2$ is contained in the set $\{
\lambda \mu : \lambda \in \sigma(V_1) \mbox{ and } \mu \in
\sigma(V_2) \}$. Thus, given that both $V_1$ and $V_2$ have
spectral radii not greater than one, the same is true for $V_1
V_2$. Consequently, $r(B_1) \le 1$. thus by Stampfli's result,
both $V_1$ and $V_2$ are contractions.
\end{proof}

We have a structure theorem for pure tetrablock isometries to go
with the result above.

\begin{thm}

Let $\Vbar = (V_1, V_2, V_3)$ be a triple of bounded operators
on a separable Hilbert space $\mathcal H$. Then $\Vbar$ is a pure
tetrablock isometry if and only if there is a separable Hilbert space $E$,
a unitary $ U : \mathcal H \rightarrow H^2(E)$ and two bounded operators
$\tau_1$ and $\tau_2$ on $E$ such that

\begin{enumerate}

\item the $H^\infty$ norm of the operator valued function $ \tau_1 + \tau_2 z$ is at most $1$,

\item $V_3 = U^* M_z^E U, V_2 = U^* M_{\varphi_2}^E U$ and $V_1 = U^* M_{\varphi_1}^E U$
where $\varphi_1(z) = \tau_1 + \tau_2^* z$ and $\varphi_2(z) = \tau_2 + \tau_1^* z$,

\item $\tau_1 \tau_2 = \tau_2 \tau_1$ and $[\tau_1 , \tau_1^* ] = [\tau_2 , \tau_2^*]$.

\end{enumerate}

\end{thm}

\begin{proof}

First suppose we are given a separable Hilbert space $E$,
a unitary $ U : \mathcal H \rightarrow H^2(E)$ and two bounded operators
$\tau_1$ and $\tau_2$ on $E$ such that

\begin{enumerate}

\item the $H^\infty$ norm of the operator valued function $ \tau_1 + \tau_2 z$ is at most $1$,

\item $V_3 = U^* M_z^E U, V_2 = U^* M_{\varphi_2}^E U$ and $V_1 = U^* M_{\varphi_1}^E U$
where $\varphi_1(z) = \tau_1 + \tau_2^* z$ and $\varphi_2(z) = \tau_2 + \tau_1^* z$,

\item $\tau_1 \tau_2 = \tau_2 \tau_1$ and $[\tau_1 , \tau_1^* ] = [\tau_2 , \tau_2^*]$.

\end{enumerate}

In that case, $V_3$ is an isometry, $V_2$ is a contraction and $V_1$ is the same $V_2^*V_3$.
Moreover, the condition (3) above implies that $\Vbar$ is a commuting triple.
Consequently, we can invoke part (3) of Theorem \ref{ti} to conclude that $\Vbar$ is a
tetrablock isometry. Moreover, the pureness of $V_3$ now implies that $\Vbar$ is a pure tetrablock isometry.

Conversely, let $\Vbar$ be a pure tetrablock isometry. The existence of $E$ is
due to the fact that $V_3$ is a pure isometry and hence is necessarily isomorphic
to multiplication by $z$ on $H^2(E)$ for some $E$.

By the commutativity of $V_1$ and $V_2$ with $V_3$, we have that
$$ V_1 =M_{\varphi_1}^{E} \mbox{ and } V_2 =M_{\varphi_2}^{E}$$
for some $\varphi_1$ and $\varphi_2$ in $H^{\infty}(E)$. Because $\Vbar$ is a tetrablock isometry, it satisfies
$V_1 = V_2^* V_3$. This, when translated in terms of the functions,
by using the power series expansion of the holomorphic functions $\varphi_1$ and $\varphi_2$,
gives us that the functions necessarily have to be of the form
$$ \varphi_1(z) = \tau_1 + \tau_2 z \mbox{ and } \varphi_2(z) = \tau_2^* + \tau_1^* z.$$
Their $H^\infty$ norms do not exceed $1$ because $V_1$ and $V_2$ are contractions.
Moreover, commutativity of $V_1$ and $V_2$ now gives the condition (3) above.
\end{proof}

After deciphering the structure of two special kinds of tetrablock contractions, we are ready to construct
the diltion.

\section{Dilation}

As we have noted already,  given a tetrablock contraction, the
dilation triple $(\Delta_1, \Delta_2, \Delta_3)$ needs to be a
tetrablock unitary. Note that a tetrablock contraction has a
tetrablock unitary dilation if and only if it has a tetrablock
isometric dilation. This is true because a tetrablock isometry is
nothing but the restriction of a tetrablock unitary to a joint
invariant subspace. In other words, a tetrablock isometry can, by
definition, be extended to a tetrablock unitary. It is elementary
that extension of a dilation is a dilation. Thus if a tetrablock
contraction has a tetrablock isometric dilation, then it also has
a tetrablock unitary dilation. We shall first construct a dilation
assuming that $(A, B, P)$ is a tetrablock contraction whose
fundamental operators $F_1$ and $F_2$ satisfy the conditions
$[F_1, F_2] = 0$ and $[F_1 , F_1^* ] = [F_2 , F_2^* ]$. This is
akin to Sch\"{a}ffer's construction of minimal isometric dilation
of a contraction.

\begin{thm}

Let $(A, B, P)$ be a tetrablock contraction on
$\mathcal H$ with fundamental operators $F_1$ and $F_2$ . Let $\mathcal D_P$ be the closure of the range of $D_P$.
Let $\mathcal K = \mathcal H \oplus \mathcal D_P \oplus \mathcal D_P
\oplus \cdots = \mathcal H \oplus l^2(\mathcal D_P) $. Consider the
operators $V_1, V_2$ and $V_3$ defined on $\mathcal{K}$ by
\begin{align*} &
V_1(h_0,h_1,h_2,\dots)=(Ah_0,F_2^* D_P h_0 + F_1 h_1 , F_2^*h_1 + F_1 h_2 , F_2^*h_2 + F_1 h_3,\dots)\\
& V_2(h_0,h_1,h_2,\dots)=(Bh_0 , F_1^* D_P h_0 + F_2 h_1 , F_1^*h_1 + F_2 h_2 , F_1^*h_2 + F_2 h_3,\dots)\\
& V_3(h_0,h_1,h_2,\dots)=(Ph_0, D_P h_0,h_1,h_2,\dots).
\end{align*}
Then \begin{enumerate}
\item $\Vbar = (V_1,V_2,V_3)$ is a minimal tetrablock isometric
    dilation of $(A, B, P)$ if $[F_1 , F_2] = 0$ and $[F_1 , F_1^* ] = [F_2 , F_2^* ]$.
\item If there is a tetrablock isometric dilation $\Wbar =
    (W_1,W_2,W_3)$ of $(A, B, P)$ such that $W_3$ is the minimal isometric dilation of $P$,
    then $\Wbar$ is unitarily equivalent to $\Vbar$. Moreover, $[F_1, F_2] = 0$ and $[F_1 , F_1^* ]
    = [F_2 , F_2^* ]$.
\end{enumerate}

\end{thm}

{\em Proof of (1)}:
It is evident from the definition that $V_3$ on $\mathcal{K}$ is
the minimal isometric dilation of $P$ in the Sch\"{a}ffer form.
Sch\"{a}ffer wrote down the unitary dilation in \cite{Schaeffer} and
we only have the isometry part of it here.

Obviously the adjoints of the three operators on $\mathcal{K}$ are
\begin{align*}
&
V_1^*(h_0,h_1,h_2,\dots)=(A^*h_0+{D_P} F_2h_1, F_1^*h_1 + F_2h_2, F_1^*h_2 + F_2h_3,\dots)\\
&
V_2^*(h_0,h_1,h_2,\dots)=(B^*h_0+{D_P} F_1h_1, F_2^*h_1 + F_1h_2, F_2^*h_2 + F_1h_3,\dots)\\
& V_3^*(h_0,h_1,h_2,\dots)=(P^*h_0+{D_P} h_1,h_2,h_3,\dots).
\end{align*}
The space $\mathcal{H}$ can be embedded inside $\mathcal{K}$ by
the map $h\mapsto (h,0,0,\dots)$. It  is clear that $\mathcal{H}$,
considered as a subspace of $\mathcal{K}$ is co-invariant under
$V_1, V_2$ and $V_3$. Moreover, $V_1^*|_{\mathcal{H}}=A^*, V_2^*|_{\mathcal{H}}=B^*$ and $V_3^*|_{\mathcal{H}}=P^*$.
This of course immediately implies (\ref{dilation}). The job now is to show that $\Vbar$ is a tetrablock isometry.

Since $V_3$ is an isometry, in order to show that $\Vbar$ is a tetrablock isometry, one has to justify the
following:
\begin{enumerate}
\item $\Vbar$ is a commuting triple,
\item $V_1 = V_2^* V_3$,
\item $r(V_1) \le 1$ and $r(V_2) \le 1$.
\end{enumerate}
If we can show these, then by part (4) of Theorem \ref{ti},
$\Vbar$ will be a tetrablock isometry.
\begin{align*}
& V_1V_3(h_0,h_1,h_2,\dots)\\
= \; & V_1(Ph_0, D_P h_0,h_1,h_2,\dots)\\
= \; & (APh_0,F_2^*D_P P h_0 + F_1 D_P h_0 , F_2^* D_P h_0 + F_1 h_1 , F_2^*h_1+ F_1 h_2, F_2^*h_2 + F_1 h_3,\dots).
\end{align*}
\begin{align*}
V_3V_1(h_0,h_1,h_2,\dots)&=V_3(Ah_0,F_2^* D_P h_0 + F_1 h_1, F_2^*h_1 + F_1 h_2, F_2^*h_2 + F_1 h_3,\dots)\\
&=(P A h_0, D_P A h_0, F_2^* D_P h_0 + F_1 h_1, F_2^*h_1 + F_1 h_2, F_2^*h_2 + F_1 h_3,\dots).
\end{align*}
Thus, to show that $V_1$ and $V_3$ commute, we only need to show
that $D_PA = F_2^*D_PP + F_1 D_P$. Similarly, for $V_2$ and $V_3$ to
commute, the criterion is that $D_PB = F_1^*D_PP + F_2 D_P$. The
proofs of these identities will use the formula (\ref{relation})
for the $\Gamma$-contraction $(A +zB , zP)$ and its
fundamental operator $F_1 + zF_2$ where $z$ is on the unit circle.
We get
$$ D_P(A + zB) = (F_1 + zF_2)^* D_P z P + (F_1 + zF_2)D_P = F_2^*D_PP + F_1D_P + z(F_1^* D_PP + F_2 D_P).$$
This holds for every $z$ on the unit circle. Therefore, the required criteria for
commutativity of $V_1$ and $V_3$ are fulfilled. The commutativity of $V_1$ and $V_2$ is more difficult.

\begin{align*}
V_1V_2(h_0,h_1,h_2,\dots) & =V_1(Bh_0, F_1^* D_P h_0 F_2h_1, F_1^* h_1 + F_2h_2, F_1^* h_2 + F_2 h_3, \dots)\\
& = (ABh_0, (F_2^*D_P B + F_1 F_1^* D_P) h_0 + F_1 F_2 h_1 ,  F_2^* F_1^* D_P h_0 \\
& + (F_2^* F_2 + F_1 F_1^*) h_1 , F_2^* F_1^* h_1+ (F_2^* F_2 + F_1 F_1^*) h_2, \dots)
\end{align*}
and
\begin{align*}
V_2 V_1 (h_0,h_1,h_2,\dots)&= V_2(Ah_0, F_2^* D_P h_0 F_1h_1, F_2^* h_1 + F_1h_2, F_2^* h_2 + F_1 h_3, \dots)\\
& =(B Ah_0, (F_1^*D_P A + F_2 F_2^* D_P) h_0 + F_2 F_1 h_1 ,  F_1^* F_2^* D_P h_0 \\
& + (F_1^* F_1 + F_2 F_2^*) h_1 , F_1^* F_2^* h_1 + (F_1^* F_1 + F_2 F_2^*) h_2, \dots).
\end{align*}

Thus, to show that $V_1$ and $V_2$ commute, we need
\begin{enumerate}
\item $F_1$ and $F_2$ commute,
\item $F_2^* F_2 + F_1 F_1^* = F_1^* F_1 + F_2 F_2^*$ and
\item $F_1^*D_PA + F_2F_2^*D_P = F_2^*D_PB + F_1F_1^*D_P$.
\end{enumerate}
The first two are part of assumption. For the third one, we have to prove that
$$F_1^*D_PA + F_2F_2^*D_P = F_2^*D_PB + F_1F_1^*D_P.$$
Or, in other words,
$$D_PF_1^*D_PA + D_PF_2F_2^*D_P = D_PF_2^*D_PB + D_PF_1F_1^*D_P$$
since the ranges of $F_1$ and $F_2$ are contained in $\mathcal D_P$. This last thing is the same as
$$ (A - B^* P)^*A - (B - A^* P)^*B = D_P(F_1 F_1^* - F_2F_2^*)D_P$$
by virtue of fundamental equations. By commutativity of $A$ and $B$, this is the same as
$$ A^* A - B^* B = D_P(F_1 F_1^* - F_2F_2^*)D_P.$$
In view of $F_1F_2 = F_2 F_1$ and Corollary \ref{TandF}, this is
equivalent to
$$ D_P(F_1^* F_1 - F_2^* F_2)D_P = D_P(F_1 F_1^* - F_2F_2^*)D_P$$
which we know to be true from part (2) above. Having gotten commutativity of the $V_i$,
we show that $V_2^* V_3$ is $V_1$. This is a straightforward computation.
\begin{align*}
V_2^*V_3(h_0,h_1,h_2,\dots)&=V_2^*(Ph_0,D_P h_0,h_1,h_2,\dots)\\
&=(B^*Ph_0+ D_P F_1 D_P h_0, F_2^* D_P h_0 + F_1 h_1, F_2^*h_1 + F_1h_2, F_2^*h_2 + F_1h_3,\dots).
\end{align*}
By the first fundamental equation, we have
$B^* P + D_P F_1 D_P = A$. Therefore we have
$V_2^* V_3 = V_1$.

We now show that $r(V_1) \leq 1$ and $r(V_2) \leq 1$.
It is clear from the definition that $V_1$ has the matrix form
$$V_1= \begin{pmatrix} A&0&0&0&\dots\\
F_2^* D_P & F_1& 0 & 0 &\dots \\
0 & F_2^* & F_1 & 0 & \dots \\
0 & 0 & F_2^* & F_1 & \dots\\
\dots&\dots&\dots&\dots&\dots\\
\end{pmatrix},$$ with respect to the decomposition $\mathcal{H}\oplus \mathcal{D}_P \oplus \mathcal{D}_P \oplus
\mathcal{D}_P \oplus...$ of $\mathcal{K}_0$. Thus $V_1$ can be written as  $
\begin{pmatrix} A & 0 \\ C_1 & E_1
\end{pmatrix}$ on $\mathcal{H}\oplus l^2(\mathcal{D}_P)= \mathcal{K}_0$, where $C_1=
\begin{pmatrix}
F_2^* D_P\\0\\0\\ \vdots
\end{pmatrix} \mbox{ and } E_1=
\begin{pmatrix} F_1&0&0&\dots\\F_2^* & F_1 & 0 & \dots \\ 0& F_2^* & F_1 & \dots\\ \dots&\dots&\dots&\dots
\end{pmatrix}$, we have by Lemma 1 of \cite{hong} that $\sigma(T_A)\subseteq\sigma(A)\cup \sigma(E_1)$.
We shall be done if we show that $r(A)$ and $r(E_1)$ are not greater than $1$.
We shall show that the numerical radius of $E_1$ is not greater than $1$.
Since spectral radius is not greater than the numerical radius, we shall be done.
Let us define \begin{align*} \varphi :&
\,\mathbb{D}\rightarrow \mathcal{B}(\mathcal{D})\\ & z
\rightarrow F_1 + F_2^*z.
\end{align*}
Clearly $\varphi$ is holomorphic, bounded and continuous on the
boundary $\partial \mathbb{D}=\mathbb{T}$ of the disc. Under the Hilbert space isomorphism
which sends $ \mathcal{D} \oplus \mathcal{D} \oplus \mathcal{D} \oplus...$ to $H^2(\mathbb D)
\otimes \mathcal D$, the operator $E_1$ goes to multiplication by the function $\varphi$.
Now $w(M_\varphi) \leq \sup\{ w(\varphi(z) ) : z \in \mathbb T \}$. Let us see what $w (\varphi(z))$ is.
Recall that the numerical radius of an operator $X$ is not greater than one if and only if
the real part of the operator $zX$ is not bigger than identity for every $z$ on the unit circle see \cite{ando-nr}.
Since we know that $w(F_1 + zF_2) \leq 1$, we have that  $w(z_1 F_1 + z_2 F_2) \leq 1$
for every $z_1$ and $z_2$ on the unit circle. Thus
$$ (z_1F_1 + z_2F_2) + (z_1F_1 + z_2F_2)^* \leq 2.$$
In other words,
$$ (z_1F_1 + \zbar_2 F_2^*) + (z_1F_1 + \zbar_2F_2^*)^* \leq 2$$
which is the same as
$$z_1 (F_1 + z F_2^*) + \zbar_1 (F_1 + z F_2^*)^* \leq 2$$
for every $z$ and $z_1$ on the unit circle. And that by Ando's result again (\cite{ando-nr}),
implies that $w(F_1 + zF_2^*) \leq 1$.

\vspace*{5mm}

{\em Proof of (2)}: Suppose a tetrablock contraction $(A, B, P)$
is given on a HIlbert space $\mathcal H$. Suppose that it has a
tetrablock isometric dilation $\Wbar = (W_1,W_2,W_3)$ on a Hilbert
space $ \mathcal K$ . Suppose $W_3$ is actually the minimal
isometric dilation of $P$. Then we can obviously take
$W_3=\begin{pmatrix} P & 0 \\ C_3 & E_3
\end{pmatrix}$ with respect to the decomposition $\mathcal H\oplus l^2(\mathcal {D}_P)$
of $\mathcal K$, where $$C_3=\begin{pmatrix} D_P \\0\\0\\ \vdots
\end{pmatrix} \mbox{ from } \mathcal H \rightarrow \mathcal{D}_P
\oplus \mathcal{D}_P \oplus \mathcal{D}_P \oplus\dots \mbox{ and }
E_3=\begin{pmatrix}
0&0&0&\dots\\I&0&0&\dots\\0&I&0&\dots\\\dots&\dots&\dots&\dots
\end{pmatrix}$$
 on  $\mathcal{D}_P \oplus\mathcal{D}_P \oplus\mathcal{D}_P \oplus\dots.$
Using this special form of $W_3$ and using the fact that $W_1$ and
$W_2$ commute with $W_3$, it takes a straightforward computation
to see that $W_1$ and $W_2$ have the operator matrix forms
 $$ W_1 = \begin{pmatrix} A &0\\C_1 & E_1
\end{pmatrix} \mbox{ and } W_2 = \begin{pmatrix} B &0\\C_2 & E_2
\end{pmatrix}$$
for some $C_i$ and $E_i$ for $i=1,2$ with respect to the decomposition of $\mathcal K$ as $\mathcal H \oplus l^2 ( \mathcal D_P)$.

Here, we shall many times use the natural identification between
Hardy space $H^2(\mathcal D_P)$ of $\mathcal D_P$ valued functions
on the unit disk and $l^2(\mathcal D_P) = \mathcal D_P \oplus
\mathcal D_P \oplus \mathcal D_P \ldots $. This Hilbert space
isomorphism will be used without further mention. Under this
Hilbert space isomorphism the operator $E_3$ is the same as the
multiplication operator $M_z^{\mathcal D_P}$ on $H^2(\mathcal
D_P)$. Because $\Wbar$ is a tetrablock isometry, we use the
characterization obtained in Theorem \ref{ti} to get $W_1 = W_2^*
W_3$ and hence
$$
\begin{pmatrix} A & 0\\ C_1 & M_{{\varphi_1}}^{{{\mathcal D}_P}} \end{pmatrix}=
\begin{pmatrix} B^*& C_2^*\\ 0 & ( M_{{\varphi_2}}^{{{\mathcal D}_P}})^* \end{pmatrix}
\begin{pmatrix} P & 0 \\ C_3 & M_z^{{{\mathcal D}_P}} \end{pmatrix} =
\begin{pmatrix} B^* P + C_2^*C_3 & C_2^*M_z^{{{\mathcal D}_P}}\\
{(M_{\varphi_2}}^{{{\mathcal D}_P}})^* C_3 & (M_{{\varphi_2}}^{{{\mathcal D}_P}})^ * M_z^{{{\mathcal D}_P}}
\end{pmatrix},
$$ which gives \begin{eqnarray}\begin{cases} \label{e9}
&(\mbox{i})\; A - B^*P = C_2^*C_3\\&(\mbox{ii})\;
C_1 = (M_{{\varphi_1}}^{{{\mathcal D}_P}})^* C_3 \\&(\mbox{iii})\;
M_{{\varphi_1}}^{{{\mathcal D}_P}} = (M_{\varphi_2}^{{\mathcal D}_P})^* M_z.
\end{cases}\end{eqnarray}

From (\ref{e9})-(iii), it is clear by considering the power
 series expansions of $\varphi_1$ and $\varphi_2$ that $\varphi_1(z) = F_1 + F_2^*z$ and
 $\varphi_2(z) = F_2 + F_1^* z$ for some $F_1$ and $F_2$ in $\mathcal B(\mathcal D_P)$. Thus
 $$E_1=\begin{pmatrix} F_1 & 0 & 0 & \dots \\
 F_2^*& F_1 & 0 & \dots \\
 0& F_2^* & F_1 &\dots \\
 \dots & \dots & \dots & \dots \end{pmatrix} \mbox{ and } E_2= \begin{pmatrix} F_2 & 0 & 0 & \dots \\
 F_1^* & F_2 & 0 & \dots \\
 0& F_1^* & F_2 &\dots \\
 \dots & \dots & \dots & \dots \end{pmatrix}$$
on $\mathcal D_P \oplus \mathcal D_P \oplus \mathcal D_P \oplus\dots$.
Combining this with (\ref{e9})-(ii), we get that
$$W_1 = \begin{pmatrix}
A &0\\ E_1^*C_3 & E_3
\end{pmatrix} \mbox{ on } \mathcal H\oplus l^2(\mathcal D_P).$$
Considering the stated matrix forms of $E_1$ and $C_3$ above, we get $E_1^*C_3=\begin{pmatrix} F_2^* D_P
\\0\\0\\\vdots \end{pmatrix}.$ Hence with respect to the decomposition
$\mathcal H \oplus \mathcal D_P \oplus\mathcal D_P \oplus \dots$ of $\mathcal K_0$, we have
$$ W_1 =\begin{pmatrix} A&0&0&0&\dots\\F_2^* D_P & F_1&0&0&\dots\\
0&F_2^*&F_1&0&\dots\\0&0&F_2^*&F_1&\dots\\\dots&\dots&\dots&\dots&\dots \end{pmatrix}.$$
A similar computation gives that
$$W_2 =\begin{pmatrix} B &0&0&0&\dots\\F_1^* D_P & F_2&0&0&\dots\\
0&F_1^*&F_2&0&\dots\\0&0&F_1^*&F_2&\dots\\\dots&\dots&\dots&\dots&\dots
\end{pmatrix}.$$  Since $W_1$ commutes with $W_3$, we have
$$ \begin{pmatrix} A&0&0&0&\dots\\F_2^* D_P & F_1&0&0&\dots\\
0&F_2^*&F_1&0&\dots\\0&0&F_2^*&F_1&\dots\\\dots&\dots&\dots&\dots&\dots \end{pmatrix}
\begin{pmatrix} P&0&0&0&\dots\\ D_P & I &0&0&\dots\\
0& I & 0 &0&\dots\\0&0& I & 0 &\dots\\\dots&\dots&\dots&\dots&\dots \end{pmatrix}$$
$$ = \begin{pmatrix} P&0&0&0&\dots\\ D_P & I &0&0&\dots\\
0& I & 0 &0&\dots\\0&0& I & 0 &\dots\\\dots&\dots&\dots&\dots&\dots \end{pmatrix}
\begin{pmatrix} A&0&0&0&\dots\\F_2^* D_P & F_1&0&0&\dots\\
0&F_2^*&F_1&0&\dots\\0&0&F_2^*&F_1&\dots\\\dots&\dots&\dots&\dots&\dots \end{pmatrix}.$$
Equating the entries in the second row and first column on both sides, we have $F_2^* D_P P + F_1 D_P = D_P A$.
Similarly, from the fact that $W_2$ commutes with $W_3$, we have $F_1^* D_P P + F_2 D_P = D_P B$.
Now by Corollary \ref{twoneweqns}, we know that $F_1$ and $F_2$ have to be the
fundamental operators of the tetrablock contraction $(A, B, P)$. This immediately tells us that $\Wbar$ and $\Vbar$ are same modulo the unitaries hidden in the arguments above.
Since $W_1$ and $W_2$ commute, equating the diagonal entries of $W_1W_2$ and $W_2W_1$,
we get that $F_1 F_2 = F_2 F_1$.
Equating the subdiagonal entries of $W_1W_2$ and $W_2W_1$, we get that $[F_1, F_1^*] = [F_2 , F_2^*]$.
That completes the proof. \qed

\textsc{Acknowledgement.} The author is thankful to Dr. Sourav Pal
who pointed out a mistake in the first version of this paper.

\end{document}